\documentclass[12pt]{article}
\usepackage{latexsym,amsmath,amssymb}
\usepackage[latin1]{inputenc}
\usepackage {amsthm}
\usepackage{times}
\usepackage{amscd}
\usepackage{epsf}
\usepackage{graphicx}
\usepackage{graphics}

\newfont{\bb}{msbm10 at 12pt}

\def\d{\hbox{\bb D}}
\def\r{\hbox{\bb R}}
\def\h{\hbox{\bb H}}

\def\hr{\hbox{\bb H}^2\times\hbox{\bb R}}
\def\hc{\hbox{\bb H}^2\times \{0 \}}
\def\m{\hbox{\bb M}^2}

\def\l{\hbox{\bb L}}
\def\pt{\frac{\partial\ }{\partial t}}

\def\s{\hbox{\bb S}}
\def\o{\hbox{\bb O}}

\newcommand{\si} {\s ^1 _{\infty}}
\newcommand{\g} {g_{-1}}

\newcommand{\norm}[1]{\left\Vert #1 \right\Vert}
\newcommand{\abs}[1]{\left\vert #1 \right\vert}
\newcommand{\set}[1]{\left\{#1\right\}}
\newcommand{\meta}[2]{\langle #1,#2 \rangle }
\newcommand{\eps}{\varepsilon}
\newcommand{\parzc}{\partial_{\bar{z}}}
\newcommand{\parz}{\partial_z}
\newcommand{\ec}{\overline{E}}
\newcommand{\ac}{\overline{\alpha}}
\newcommand{\hz}{h_{z}}
\newcommand{\hzb}{h_{\overline{z}}}
\newcommand{\zb}{\overline{z}}

\newcommand{\nzb}{\nu _{\overline{z}}}

\newcommand{\bi}{\partial _{\infty}}

\numberwithin{equation} {section}

\begin{document}

\theoremstyle{plain}\newtheorem{lema}{Lemma}[section]
\theoremstyle{plain}\newtheorem{pro}{Proposition}[section]
\theoremstyle{plain}\newtheorem{teo}{Theorem}[section]
\theoremstyle{plain}\newtheorem{ex}{Example}[section]
\theoremstyle{plain}\newtheorem{remark}{Remark}[section]
\theoremstyle{plain}\newtheorem{corolario}{Corollary}[section]
\theoremstyle{plain}\newtheorem{defi}{Definition}[section]

\begin{center}
\rule{15cm}{1.5pt} \vspace{0.5cm}

{\Large \bf Complete surfaces with positive extrinsic curvature\\[2mm] in product
spaces.}

\vspace{0.5cm}

{\large José M. Espinar$\,^\dag$, José A. Gálvez$\,^\dag$, Harold Rosenberg$\,^\ddag$}\\
\vspace{0.3cm} \rule{15cm}{1.5pt}
\end{center}

\vspace{.5cm}

\noindent $\mbox{}^\dag$ Departamento de Geometría y Topología, Universidad de
Granada, 18071 Granada, Spain; e-mails: jespinar@ugr.es; jagalvez@ugr.es\vspace{0.2cm}

\noindent $\mbox{}^\ddag$ Institut de Mathématiques, Université Paris VII, 2 place
Jussieu, 75005 Paris, France; e-mail: rosen@math.jussieu.fr

\vspace{.3cm}

\begin{abstract}
We prove that every complete connected immersed surface with positive extrinsic
curvature $K$ in $\hr$ must be properly embedded, homeomorphic to a sphere or a
plane and, in the latter case, study the behavior of the end. Then, we focus our
attention on surfaces with positive constant extrinsic curvature ($K- $surfaces). We
establish  that the only complete $K-$surfaces in $\s^2\times\r$ and $\hr$ are
rotational spheres.  Here are the key steps to achieve this. First height estimates
for compact $K-$surfaces in a general ambient space $\m\times\r$ with boundary in a
slice are obtained. Then   distance estimates for compact $K-$surfaces (and
$H-$surfaces) in $\hr$ with boundary on a vertical plane are obtained. Finally we
construct a quadratic form with isolated zeroes of negative index.
\end{abstract}

\section{Introduction.}

In $1936$, J. Stoker \cite{S} generalized the result of J. Hadamard \cite{H}   that
a compact strictly locally convex surface in the Euclidean 3-space $\r ^3$ is
homeomorphic to the sphere. J. Stoker showed that a complete strictly locally convex
immersed surface in $\r ^3$ must be embedded and homeomorphic to the sphere or
plane. In the latter case, the surface is a graph over a planar domain. Today, this
result is known as the Hadamard-Stoker Theorem.  Strict convexity of the surface is
equivalent in $\r ^3$ to positive  Gaussian curvature. Note   that in $\r ^3$, the
Gauss equation for a surface says that the Gauss curvature, i.e. the intrinsic
curvature $K(I)$, and the Gauss-Kronecker curvature, i.e. the extrinsic curvature
$K$, are equal. In space forms, the situation is  close since both curvatures are
related by a constant.

M. Do Carmo and F. Warner \cite{CW} extended Hadamard's Theorem to the hyperbolic
3-space $\h ^3 $, assuming the surface is compact and has positive extrinsic
curvature. The complete case in $\h ^3 $ was treated by R. J. Currier in \cite{C}
and it is interesting to remark on the difference with the euclidean case. Currier's
Theorem says that a complete immersed surface in $\h ^3$ whose principal curvatures
are greater than or equal to one, is embedded and homeomorphic to the sphere or
plane. And we cannot expect a better result; it is easy construct examples of
complete embedded flat surfaces, i.e. $K=1$  in $\h^3$, homeomorphic to a cylinder.

Recently, the study of surfaces in product spaces $\m \times \r$, where $\m $ is a
Riemannian surface and $\r$ the real line, has undergone considerable development.
U. Abresch and H. Rosenberg \cite{AR} defined a holomorphic quadratic differential
on constant mean curvature surfaces in the homogeneous 3-manifolds.  This enabled
them to generalize Hopf's theorem to these spaces: immersed constant mean curvature
spheres are rotational and embedded. Aledo, Espinar and Galvez associated a
holomorphic quadratic differential to constant Gaussian curvature surfaces in $\s ^2
\times \r$ and $\h ^2 \times \r$,  \cite{AEG}.  This enabled them to prove the same
Hopf type theorem for immersed constant Gaussian spheres in these product spaces.
Also, they classified the complete surfaces with constant Gauss curvature, and
established Liebmann and  Hilbert type Theorems for these surfaces.  More precisely,

\vspace{.3cm}

{\bf Liebmann type Theorem.} Given a real constant $K(I)$, there exists  a unique
complete surface of constant Gaussian curvature $K(I) > 1$ in $\s ^2 \times \r$ ,
(up to isometry), and a unique complete surface of constant Gaussian curvature $K(I)
> 0$ in $\hr$. In addition, these surfaces are rotationally symmetric embedded spheres.

\vspace{.3cm}

{\bf Hilbert type Theorem.} There is no complete immersion of constant Gaussian
curvature $K(I) < -1$ into $\hr$ or $\s ^2 \times \r$.

\vspace{.3cm}

Also, there exist  complete immersions of every constant curvature $K(I) \geq -1$
into $\hr$. In \cite{AEG2} the authors prove  there are no complete immersions with
constant Gaussian curvature $0 < K(I) < 1$ in $\s ^2 \times \r$. The existence of
complete immersions with constant Gaussian curvature $-1 \leq K(I) < 0$ in $\s ^2
\times \r$ remains open.

In contrast, the case of extrinsic curvature has been rarely considered in these
spaces (see \cite{CR}). We  note that the classification of surfaces of constant
Gaussian curvature does not help us since the intrinsic and extrinsic curvature
differ by the sectional curvature function in a product space.

We center our attention on complete surfaces with positive (non constant and
constant) extrinsic curvature in $\hr$, nonetheless some of our results for constant
extrinsic curvature also work in a more general setting as we will point out.

One of the main results of this paper is that only embedded rotational spheres can
occur when $K$ is a positive constant; Theorem 7.3.

We organize the paper as follows. In Section 2 we introduce the notation and
definitions we need.   In Section 3 we establish the following Hadamard-Stoker type
Theorem in $\hr$, (the notion of simple end will be defined later).

\vspace{.3cm}

{\bf Theorem \ref{t3.1}.} Let $S$ be a complete connected immersed surface with $K
> 0$ in $\hr$. Then $S$ must be properly embedded and bounds a
strictly convex domain in $\hr$. Moreover, $S$ is homeomorphic to $\s ^2$ or $\r^2$.
In the latter case, $S$ is a graph over a convex domain of $\hc$ or $S$ has a simple
end.

\vspace{.3cm}

This result suggests that  surfaces with positive extrinsic curvature in $\hr$
behave like surfaces with $K>0$ in $\r ^3$, rather than surfaces with $K>0$ in $\h
^3$. This is because there are many totally geodesic foliations of $\hr$ by vertical
planes which are isometric to $\r ^2$.

In Section 4 we construct complete embedded surfaces with positive extrinsic
curvature with a simple end in $\hr$. In Section 5 we classify the complete
revolution surfaces of positive constant extrinsic curvature in $\hr$ which are
topological spheres. Hereafter we will refer to surfaces with positive constant
extrinsic curvature as $K-$surfaces.

In Section 6 we establish vertical height estimates for $K-$surfaces in $\m \times
\r$, $\m $ a Riemannian surface. More precisely,

\vspace{.3cm}

{\bf Theorem \ref{testver}.} Let $\psi : S \longrightarrow \m \times \r$ be a
compact graph on a domain $\Omega \subset \m $, with positive constant extrinsic
curvature $K$ and whose boundary is contained in the slice $\m \times \set{0}$. Let
$k$ be the minimum of the Gauss curvature on $\Omega \subset \m $. Then, there
exists a constant $c_K $ (depend only on $K$ and $k$) such that $h(p) \leq c_K$ for
all $p \in S$, ($h$ is the height function on the graph).

\vspace{.3cm}

Also, horizontal height (or distance) estimates are obtained,

\vspace{.3cm}

{\bf Theorem \ref{testhor}.} Let $S$ be a compact embedded surface in $\hr$, with
extrinsic curvature a constant $K>0$. Let $P$ be a vertical plane in $\hr$ and
assume $\partial S \subset P$. Then the distance from $S$ to $P$ is bounded; i.e.,
there is a constant $d$, independent of $S$, such that
$$ \mbox{dist}(q , P) \leq d \, , \quad \forall q \in S .$$

\vspace{.3cm}

We remark that the proof of this result works for $H-$surfaces in $\hr$ with $H>
1/2$, thus this result, together with the vertical height estimates given in
\cite{AEG1} for $H-$surfaces with $H>1/2$, generalizes \cite[Theorem 1.1]{NR} for
$H-$surfaces with $H>1/2$. More precisely,

\vspace{.3cm}

{\bf Theorem \ref{lnor2}.} For $K>0$ (or $H > 1/2$) there is no properly embedded
$K-$surface ($H-$surface) in $\hr$ with finite topology and one end.

\vspace{.3cm}

Finally, in Section 7 we classify the complete immersed $K-$surfaces in $\hr$ and
$\s ^2 \times \r$.

\vspace{.3cm}

{\bf Theorem \ref{tKconstant}.} The complete immersions with positive constant
extrinsic curvature $K$ in $\hr$ and $\s ^2 \times \r$ are the rotational spheres
given in Section 5.

\section{Notations.}

In Sections 2, 3 and 4 we will use the \emph{Poincaré disk model} of $\h ^2$. In
Section 5 we will work in  the hyperboloid of one sheet model of  $\h^2$ in $\l ^3$.
We make this precise in Section 5.

In the Poincaré model, $\h ^2$ is represented as the domain
$$\d =\set{z\equiv (x,y)
\in \r ^2 \, : |z|^2= x^2 + y^2 < 1}$$ endowed with the metric $\g=
\dfrac{4|dz|^2}{(1-|z|^2)^2}$.

The complete geodesics in this model are given by arcs of circles or  straight lines
which are orthogonal to the boundary at infinity
$$\s ^1 _{\infty} = \set{z \in \r^2 \, : |z| =1}.$$

Thus, the asymptotic boundary of a set $\Omega \subset \h^2$ is
$$ \bi \Omega = \mbox{cl}(\Omega ) \cap \s ^1 _{\infty} $$
where $\mbox{cl}(\Omega)$ is the closure of $\Omega$ in $\set{z \in \r ^2 \, :
|z|\leq 1}$.

We  orient $\h ^2 $ so that its boundary
at infinity is oriented counter-clockwise. Let $\gamma $ be a complete oriented geodesic in
$\h^2$, then
$$ \bi \gamma = \set{ \gamma ^{-} , \gamma ^{+} } $$where
$\gamma ^{-} = \lim _{t \rightarrow - \infty}\gamma (t)$ and $\gamma ^{+} = \lim _{t
\rightarrow + \infty}\gamma (t)$.   Here $t$ is  arc length along $\gamma$. We will
often identify a  geodesic $\gamma $ with its boundary at infinity, writing
\begin{equation}\label{2.4}
\gamma = \set{\gamma ^{-} , \gamma ^{+}}
\end{equation}

\begin{defi}\label{d2.1}
Let $\theta _1 , \theta _2 \in \si $, we define the oriented geodesic joining
$\theta _1 $ and $\theta _2 $, $\gamma (\theta _1 , \theta _2)$, as the oriented
geodesic from $\theta _1 \in \si$ to $\theta _2 \in \si$.  Here we represent points
on the circle as real numbers (angles) by their image by the exponential map.
\end{defi}

We observe that given an oriented geodesic $\gamma =\set{\gamma ^{-} ,\gamma ^{+}}$
in $\h ^2$ then $\h ^2 \setminus \gamma $ has two connected components. We will
distinguish them using the following notation,
\begin{defi}\label{d2.2}
Let $J$ be the standard counter-clockwise rotation operator. We call exterior set of
$\gamma$ in $\h^2$, $\mbox{ext}_{\h^2} (\gamma )$, the connected component of $\h^2
\setminus \gamma$ towards which  $J \gamma ' $ points. The other connected component
of $\h^2 \setminus \gamma$ is called the interior set of $\gamma $ in $\h ^2$ and
denoted by $\mbox{int}_{\h^2} (\gamma )$.
\end{defi}

On the other hand, we consider the product space $\hr$ represented as the domain
$$ \hr =\set{ (x,y ,t) \in \r ^3 \, : \,  x^2 + y^2 < 1}$$
endowed with the product metric $\meta{}{} = g_{-1} + dt ^2$. In addition, we denote
by $\pi : \hr \longrightarrow \hc $ the usual projection and $\pt$ the gradient of
the function $t$ in $\hr$.

Given a complete oriented geodesic $\gamma$ in $\hc$, we will call $\gamma \times
\r$ a \emph{vertical plane} of $\hr$ and we will call a slice $\h ^2 \times \{ \tau
\}$ a \emph{horizontal plane}. Note that a vertical plane is isometric to $\r^2$ and
a horizontal plane is isometric to $\h ^2$.

The notions of the interior and exterior domains of a horizontal oriented geodesic
extend naturally to vertical planes.
\begin{defi}\label{d2.3}
For a complete oriented geodesic $\gamma$ in $\hc\equiv\h^2$ we call, respectively,
interior and exterior of the vertical plane  $P=\gamma \times \r$ the sets
$$
\mbox{int}_{\hr}(P)= \mbox{int}_{\h ^2}(\gamma) \times \r,\qquad
\mbox{ext}_{\hr}(P)= \mbox{ext}_{\h ^2}(\gamma) \times \r.
$$
\end{defi}

We will often use  foliations by vertical planes of  $\hr$.  We now make this precise.

\begin{defi}\label{d2.4}
Let $P$ be a vertical plane in $\hr$ and let $\gamma (t)$ be an oriented horizontal
geodesic in $\hc$, with $t$   arc length along $\gamma$,   $\gamma (0) = p_0 \in P$,
$\gamma ' (0)$ orthogonal to $P$ at $p_0$ and $\gamma (t) \in \mbox{ext}_{\hr}(P)$
for $t>0$. We define the \emph{oriented foliation of vertical planes along
$\gamma$}, denoted  by $P_{\gamma}(t)$, to be the vertical planes orthogonal to
$\gamma (t)$ with $P = P_{\gamma}(0)$.
\end{defi}

To finish, we will give a definition of a particular type of curve in a vertical plane.
\begin{defi}\label{d2.6}
Let $P$ be a vertical plane and $\alpha$ a complete embedded convex curve in $P$. We
say that $\alpha$ is \emph{vertical} (in $P$) if there exist a point $p \in \alpha
$, called a \emph{vertical point}, and a vertical direction  $v = \pm \pt $, such
that the half-line $p + s v $, $s > 0$, is contained in the convex body bounded by
$\alpha$ in $P$.
\end{defi}

\section{A Hadamard-Stoker type theorem.}

This section is devoted to the proof of a Hadamard-Stoker type Theorem in $\hr$. Let
us consider a surface $S $ and  $\psi:S \longrightarrow \hr $ an immersion with
positive extrinsic curvature $K = \mbox{det}(II)/ \mbox{det}(I)$, where $I$ and $II$
are  the first and second fundamental forms of $S$.

Observe that the definition of $K$ does not depend on the local choice of a unit
normal vector field $N$. Nevertheless, $N$ can be globally chosen since $K>0$, that
is, $II$ is definite. From now on we will identify $\psi (S)$ with $S$.

We begin with an elementary, but useful, result.
\begin{pro}\label{p3.1}
Let $S$ be an immersed surface with positive extrinsic curvature in $\hr$. Let $P$ be a
horizontal or vertical plane in $\hr$. If $S$ and $P$ intersect transversally then
each connected component $C$ of $S \cap P$ is a strictly convex curve in $P$.
\end{pro}
\begin{proof}
Let us parametrize $C$ as $\alpha (t)$ where $t$ is  arc length. Then since $P$ is a
totally geodesic plane, we have
\begin{equation*}
\begin{split}
\nabla ^P _{\alpha '}\alpha ' &= \bar{\nabla } _{\alpha '}\alpha ' = \nabla ^S
_{\alpha '} \alpha ' + II (\alpha ' ,\alpha ' ) N
\end{split}
\end{equation*}where $\nabla ^P$, $\bar{\nabla} $ and $\nabla ^S$ are the connections on $P$,
$\hr$ and $S$ respectively. Since  the extrinsic curvature is positive we
have $II(\alpha ',\alpha ') \neq 0$. Thus $\nabla ^P _{\alpha '}\alpha ' \neq 0$, that is, the
geodesic curvature of $C$ vanishes nowhere on $P$.
\end{proof}

\begin{defi}
Let $S \subset \hr$ be a surface. We say that $S$ has a simple end if the boundary
at infinity of $\pi(S)\subset \hc\equiv\h^2$ is a unique point
$\theta_0\in\s^1_{\infty}$ and, in addition, for all
$\theta_1,\theta_2\in\s^1_{\infty}\setminus\{\theta_0\}$ the intersection of the
vertical plane $\gamma(\theta_1,\theta_2)\times\r$ and $S$ is empty or a compact
set.
\end{defi}

Now, we can establish the main Theorem of this section.

\begin{teo}\label{t3.1}
Let $S$ be a complete connected immersed surface in $\h^2\times\r$ with $K > 0$.
Then $S$ must be properly embedded and bounds a strictly convex domain in $\hr$.
Moreover, $S$ is homeomorphic to $\s ^2$ or $\r^2$. In the latter case, $S$ is a
graph over a convex domain of $\hc$ or $S$ has a simple end.
\end{teo}
\begin{proof}
We first distinguish two cases, depending on the existence of a point on $S$ with
horizontal unit normal, or equivalently, depending on the existence of a vertical
tangent plane.\newline

Suppose there is no point $p \in S$ with a vertical tangent plane at $p$. We will
show $S$ is a graph and homeomorphic to $\r ^2$.

Let $P$ be a vertical plane which meets $S$ transversally. Let $\gamma$ be an
oriented horizontal geodesic orthogonal to $P$ and consider the foliation
$P_{\gamma}(t)$
 of vertical planes along $\gamma$ (see Definition \ref{d2.4}). Now,
if $P_{\gamma}(t)\cap S \neq \emptyset$, using that there is no point $p \in S$ with
a vertical tangent plane at $p$ and Proposition \ref{p3.1}, each connected component
of $P_{\gamma}(t) \cap S $ is a non-compact complete embedded strictly convex curve.
Otherwise, if a connected component has a self-intersection or it is compact, then
it has a point with a vertical tangent line, which means that $S$ has a point with a
vertical tangent plane at that point.

Let $C(0)$ be an embedded component of $P\cap S=P_{\gamma}(0)\cap S$. Let us
consider how $C(0)$ varies as $t$ increases to $+\infty$. No two points of
$P_{\gamma}(t_0) \cap S$ can join at some $t_0 >0$, since this would produce a
vertical tangent plane at some point. So the component $C(0)$ of $P_{\gamma}(0)$
varies continuously to one embedded curve $C(t)$ of $P_{\gamma}(t) \cap S$ as $t$
increases. The only change possible is that $C(t)$ goes to infinity as $t$ converges
to some $t_1$ and disappears in $P_{\gamma}(t_1)$.

Similarly $C(0)$ varies continuously to one embedded curve of $P_{\gamma}(t)\cap S$
as $t \longrightarrow  -\infty$. Hence $S $ connected yields $P_{\gamma}(t) \cap S$
is at most one component for all $t$. So, we conclude $S$ is a vertical graph. To
finish, we observe that $P_{\gamma}(t) \cap S$ is empty or homeomorphic to $\r$ for
each $t$, hence $S$ is topologically $\r ^2$.\newline

Now, for the rest of the proof we suppose there is a point $p_0 \in  S$ with a
vertical plane $P$ tangent at $p_0$. We will show $S$ is homeomorphic to $\s ^2$ or
$S$ is homeomorphic to $\r ^2$ and has a simple end.

By assumption, it is easy to see that there exist neighborhoods $p_0 \in U \subset S
$ and $V \subset P$ such that $U$ is a horizontal graph over $V$. Also, because of
$K(p_0)>0$, $S$ is strictly locally convex at $p_0$, hence we can assume that $U
\subset S$ is on one side of $P$. Let $P_{\gamma}(t)$ be the foliation of vertical
planes along $\gamma $, being $\gamma$ a horizontal geodesic with $\gamma (0) = p_0$
and $\gamma ' (0)$ orthogonal to $P$. Note that, up to an isometry, we can suppose
that $U \setminus\{ p_0\} \subset \mbox{ext}_{\hr} (P)$, $p_0 \in \hc$ and $\gamma =
\set{\pi/2 , 3\pi /2}$.

From Proposition \ref{p3.1} and the fact that locally $S$ is a graph, there exist
$\eps > 0$ such that $P_{\gamma}(t) \cap U $ are embedded compact strictly convex
curves for all $0 < t < \eps$. For $0< t < \eps$, let $C(t)$ denote the connected
component of $P_{\gamma}(t)\cap S$ which coincides with $P_{\gamma}(t)\cap U$.
Perhaps $P_{\gamma}(t)\cap S $ has other components distinct from $C(t)$ for each
$0< t < \eps$, but we only care how $C(t)$ varies as $t$ increases. We also denote
by $C(t )$ the continuous variation of the curves $P_{\gamma} ( t)\cap S$, when $t >
\varepsilon$.

We distinguish two cases:

\begin{enumerate}
\item[A)] $C(t)$ remains compact as $t$ increases:

By topological arguments it is easy to show that, if  $C(t)$ remains compact and
non-empty as $t$ increases, then the $C(t)$ are embedded compact strictly convex
curves or a point.
\begin{enumerate}
\item[A.1)] If $C(t)$ remains compact and non-empty as $t \rightarrow +\infty$, then since
$S$ is connected, $S$ must be embedded. In addition, because $C(0)$ is a point and
$C(t)$ is homeomorphic to a circle for every positive $t$, $S$ is homeomorphic to
$\r^2$.

Now, from the fact that $C(t)$ remains compact, then
$$
\partial _{\infty} \pi(S) = \{3\pi/2\}\subset\s^1_{\infty}
$$
and $S$ has a simple end.

\item[A.2)] If there exists $\bar{t} > 0 $ such that $C(t)$ are compact for all
$0 < t < \bar{t}$ and the component $C(t)$ disappears for $t > \bar{t}$, then, using
that $S$ is connected, $S$ is either compact, embedded and topologically $\s ^2$ or
non compact, embedded and topologically $\r^2$. That is, if the $C(t)$ converge to a
compact set as $t$ converges to $\bar{t}$ then $C(\bar{t})$ must be a point (because
our surface has no boundary) and $S$ is a sphere. Otherwise the $C(t)$ drift off to
infinity as $t$ converges to $\bar{t}$ and $S$ is topologically a plane.

We now show that in the latter case, the vertical projection $\pi$ of $S$ has
asymptotic boundary one of the two points at infinity of $\pi(P_{\gamma}(\bar{t}))$.

Without lost of generality we can assume that $P_{\gamma}(\bar{t})= \beta \times \r$
where $\beta = \set{\beta ^{-}, \beta ^{+}}$. Consider the vertical plane $Q= \gamma
\times \r$. Let $\tilde{C}$ be the component of $Q \cap S$ containing $p_0$. First
observe that $\tilde{C}$ is compact, otherwise it would intersect the line $Q \cap
P_{\gamma}(\bar{t})$ in two points, which is not the case. Thus, we can consider the
disk $\tilde{D}$ bounded by $\tilde{C}$ on $S$.

Let $Q_{\beta}(t)$ denote the foliation by vertical planes along $\beta$,
$Q_{\beta}(0)=Q$. There exists $t_0$ (we can assume $t_0<0$) satisfying
$Q_{\beta}(t_0)$ touches $\tilde{D}$ on one side of $\tilde{D}$ by compactness. Let
$q_0\in \tilde{D}\cap Q_{\beta}(t_0)$ be the point where they touch. Consider the
variation $\tilde{C}(t)$ of $q_0$ on $S\cap Q_{\beta}(t)$ from $t=t_0$ to infinity.
Then, $\tilde{C}(t)$ is a convex embedded curve for $t$ in a maximal interval $
(t_0,\bar{t}_0)$ with $0<\bar{t}_0\leq\infty$. Hence, $S$ is foliated by the
$\tilde{C}(t)$, $\tilde{C}=\tilde{C}(0)=Q \cap S$ and
$\beta^-\not\in\partial_{\infty}\pi(S)$ because $S$ is on one side of
$Q_{\beta}(t_0)$.

Now, we will show that $\partial_{\infty}\pi (S)=\set{\beta ^{+}}$. Let $\gamma
(\theta) $ denote the complete horizontal geodesic starting at $p_0$ and making an
angle $\theta$ with $\gamma $ at $p_0$. Assume $\gamma (\theta)$ enters the side of
$Q$ containing $\beta ^{+}$, for $0 < \theta <\pi/2$. Let $\bar{\theta}$ be the
value of $\theta$ such that $\gamma (\bar{\theta})$ is asymptotic to $\beta ^{+}$.
Let $Q(\theta) = \gamma (\theta ) \times \r$. For each $\theta$,
$0\leq\theta<\bar{\theta}$, we have $S \cap Q(\theta)$ is one connected embedded
compact curve $C'(\theta)$. The proof of this is the same as the previous one for
$\tilde{C}$. Notice that each $C'(\theta)$ is non empty since $p_0 \in C'(\theta)$.

Now $C'(\bar{\theta})$ can not be compact, otherwise $S$ could not be asymptotic to
the plane $P_{\gamma} (\bar{t})$, a contradiction.

In order to complete the proof of the Case A.2 we show that $S$ has a simple end.
Observe that $C'(\theta)$ is compact, $\bar{\theta}<\theta<\pi/2$ because
$S=\cup_{0\leq t<\bar{t}} C(t)$. Moreover, $C'(\theta)\subset \tilde{D}$,
$-\pi/2<\theta<0$, and $ \tilde{D}$ is compact. Thus, it is easy to conclude that
$S$ has a simple end.

Thus we have proved that in Case A.2, $S$ is either a properly embedded sphere or
$S$ is a properly embedded plane with a simple end at $\beta ^{+}$.
\end{enumerate}

\item[B)] $C(t)$ becomes non-compact:

Let $\bar{t} >0$ be the smallest $t$ with $C(\bar{t})$ non-compact, $C(\bar{t})$ the
limit of the $C(t) $ as $t \rightarrow \bar{t}$, $C(\bar{t})$ is an embedded
strictly convex curve in $P_{\gamma}(\bar{t})$.

\textbf{Claim 1:} We now show $C(\bar{t})$ is not vertical (see Definition
\ref{d2.6}).

Let us assume that $C(\bar{t})$ is vertical, and let $q \in C(\bar{t})$ be a
vertical point. First of all, note that $ \tilde{S} = \bigcup _{0\leq t \leq
\bar{t}} C(t) \subset S$ is embedded. Let $q ' = \pi (q) \in \hc $ and let us
consider $\Gamma _{p_0 q' }$ the complete horizontal geodesic joining $p_0$ and $q
'$. Let $Q = \Gamma _{p_0 q '} \times \r$, and consider $r_0 = Q \cap P_{\gamma}
(0)$ and $r_{\bar{t}} = Q \cap P_{\gamma}(\bar{t})$. Note that $r_0$ and
$r_{\bar{t}}$ are parallel lines in $Q$. Also, $\alpha _Q = Q \cap \tilde{S}$ is a
non-compact embedded strictly convex curve in $Q$ such that $r_0$ is tangent to
$\alpha _Q$ at $p_0 \in \alpha _Q $ and $\alpha _Q \cap r_{\bar{t}}$ is exactly one
point, since $C(\bar{t})$ is vertical. But this is a contradiction because $\alpha
_Q$ is a strictly convex curve in $Q$, which is isometric to $\r ^2$, and it must
intersect $r_{\bar{t}}$ twice.

Thus, $C(\bar{t})$ is not vertical, and we claim that

\textbf{Claim 2:} $\bi \pi (C(\bar{t}))$ is one point.

Let us denote by $D(t)$ the convex body bounded by $C(t)$ in $P_{\gamma}(t)$ for
each $0 < t < \bar{t}$. Thus, the limit, $D(\bar{t})$, of $D(t)$ as $t$ increases to
$\bar{t}$ is an open convex body bounded by $C(\bar{t})$ in $P_{\gamma}(\bar{t})$
which is isometrically $\r ^2$. If $\bi \pi (C(\bar{t}))$ has two points, the only
possibility is that $C(\bar{t})$ is vertical, which is impossible by Claim 1.

Let $\delta _0> 0 $ and $t_{\delta _0} < \bar{t}$ such that $P(t_{\delta _0}) =
\Gamma (\delta _0) \times \r$ where $\Gamma (\delta _0) = \set{\delta _0, \pi -
\delta _0}$. We denote by $\tilde{S}_1 =\bigcup _{0\leq t \leq t_{\delta _0}}
C(t)\subset S$ and note that $\tilde{S}_1$ is connected and embedded.

Let us consider the complete horizontal geodesic given by $\Gamma (\delta _0
,s)=\set{\delta _0, \pi -\delta _0+ s}$ and the vertical plane $Q(s)=\Gamma (\delta
_0,s)\times \r$ for each $s \geq 0$. So, $Q(0) = P_{\gamma}(t_{\delta _0})$ and
$Q(0) \cap \tilde{S}_1 = C(t_{\delta _0})$ is an embedded compact strictly convex
curve. Let us consider how $\alpha (s) = Q(s) \cap S$ varies as $s$ increases to
$\pi +\delta _0$. At this point, we have two cases:
\begin{enumerate}
\item[B.1)] $\alpha (s) $ remains compact for all $ 0 \leq s < \pi + \delta _0$:

In this case, letting $\delta _0\rightarrow 0$, corresponds to Case A.1.

So, without lost of generality we can assume that $P_{\gamma}(\bar{t})= \Gamma
(\bar{t}) \times \r$ where $\Gamma (\bar{t})= \set{0 , \pi}$ and
\begin{equation}\label{3.6}
\bi \pi (C(\bar{t}))=\set{0} .
\end{equation}

\item[B.2)] $\alpha (s) $ becomes non-compact:

Let $0 <\bar{s} <\pi +\delta _0$ be the smallest $s$ with $\alpha(\bar{s})$
non-compact, $\alpha (\bar{s})$ is the limit of the $\alpha(s) $ as $s \rightarrow
\bar{s}$. Also,
\begin{equation*}
\bi \pi (\alpha (\bar{s})) = \{ \pi -\delta_0 + \bar{s} \},
\end{equation*}
otherwise it must be $\set{\delta _0}$ which contradicts (\ref{3.6}).

So, without lost of generality we can assume that $P_{\gamma}(\bar{t})= \Gamma
(\bar{t}) \times \r$ where $\Gamma (\bar{t})= \set{0 , \pi}$ and
\begin{equation}\label{3.6}
\bi \pi (C(\bar{t}))=\set{0} .
\end{equation}

Clearly $\delta_0<\bar{s}$. For each $\delta \leq \delta _0$ we consider the
complete horizontal geodesic given by $ \sigma (\delta) = \set{ \delta , \pi
+\bar{s}-\delta_0-\delta} $ and the vertical plane $T(\delta) = \sigma (\delta)
\times \r$. Let us denote by $\tilde{S} _2=\bigcup _{0\leq s \leq \bar{s}-2\delta
_0} \alpha (s) \subset S$ and note that $\tilde{S}_2$ is connected and embedded, so,
$\tilde{S}=\tilde{S}_1 \cup \tilde{S}_2 \subset S$ is connected and embedded. For
each $\delta$, $ 0 < \delta \leq \delta _0 $, $E ( \delta ) = T( \delta ) \cap
\tilde{S}$ is a strictly convex compact embedded curve in $T(\delta)$. As $\delta
\rightarrow 0$, theses curves converge to a convex curve in $T( 0 )$ with $\partial
_{\infty} \pi ( E( 0 ) )$ the two points $\{ 0 , \pi - \delta _0 +\bar{s}\}$. This
contradicts Claim 2. Hence $\alpha (s )$ can not become non-compact and we are in
the Case B.1.
\end{enumerate}
\end{enumerate}
Finally, $S$ bounds a strictly convex body follows from Proposition \ref{p3.1}, and
the fact that every geodesic in $\hr$ lies in a vertical plane. This completes the
proof of Theorem \ref{t3.1}.
\end{proof}

\section{Complete surfaces in H$^2$xR with K$>$0 and a simple end.}

This Section is devoted to the construction of some examples of complete embedded
surfaces with positive extrinsic curvature and a simple end.

Note that if $S$ has a  simple end, then for each $t \in \r$ where $S \cap (\h ^2
\times \set{t})$ is not compact, the intersection is a convex curve in $\h ^2 \times
\set{t}$ with asymptotic boundary one point. Also, the point at infinity of each
horizontal section is the same, that is, $\bi \pi(S \cap \h ^2 \times \set{t}) =
\set{\bar{\theta}} \in \si $.

Bearing this in mind, we fix a point $\bar{\theta} \in \si$, say
$\bar{\theta}=0\equiv(1,0)$, and consider the 1-parametric isometry group given by
$$
F_t(x,y,z)=\left(1+\frac{4(x-1)}{4+4ty+t^2((x-1)^2+y^2)},
\frac{4y+2t((x-1)^2+y^2)}{4+4ty+t^2((x-1)^2+y^2)},z\right).
$$
Here, the orbit of any point $p\in\h^2\times\r$ is a horocycle $H_p$ contained in a
slice such that $\partial_{\infty}\pi(H_p)=\{\bar{\theta}\}.$

Let $P$  be a plane orthogonal to every orbit, say $P=\{(x,y,z)\in\hr: y=0\}$. We
parametrize $P$ as
$$
\psi(x,y)=\left(\frac{e^x-1}{e^x+1},0,y\right)
$$
in such a way that its induced metric is $dx^2+dy^2$.

Now, let $\alpha(y)=\psi(\rho(y),y)$ be a curve on the vertical plane $P$ for a
suitable function $\rho$ and consider the helicoidal surface $S_{\rho}$ given by
$f_{\rho}(y,t)=F_t(\alpha(y))$.

Then, it is easy to check that the extrinsic curvature of $S_{\rho}$ is
$$
K=\frac{\rho''(y)}{(1+\rho'(y)^2)^2}.
$$

Thus, we obtain
\begin{pro}
Let  $\mathcal{I}= (y_1,y_2)$ be an open interval where $-\infty \leq y_1 < y_2 \leq
+\infty$, and $\rho : \mathcal{I} \longrightarrow \r $ a function such that $\rho ''
(y) > 0 $ for all $y \in \mathcal{I}$, $\lim _{y \rightarrow y_1} \rho (y)=\lim _{y
\rightarrow y_2} \rho (y) =+\infty$. Then the surface $S_{\rho}$  is a properly
embedded surface with positive extrinsic curvature and a simple end.
\end{pro}

\section{Complete revolution surfaces of constant positive extrinsic curvature.}

In this section we focus our attention on the study of the complete revolution
surfaces of positive constant extrinsic curvature in $\hr$.

Let us consider the Lorentz-Minkowski 4-space $\l^4$ with induced metric
$-dx_1^2+dx_2^2+dx_3^2+dx_4^2$. Here, we consider $\hr$ as the submanifold of $\l^4$
given by
$$
\hr=\{(x_1,x_2,x_3,x_4)\in \l^4: -x_1^2+x_2^2+x_3^2=-1,\ x_1>0\}.
$$

It is well known that the special orthogonal group $\s\o(2)$ can be identified with
the subgroup of isometries of $\hr$ (rotations) which preserves the orientation and
fixes an axis $\{p\}\times\r$, with $p\in\hc$.

Up to isometries, we can assume that the axis is given by $\{(1,0,0)\}\times\r$. In
such a case, the set ${\cal P}=\{(x_1, x_2, x_3,x_4)\in\hr :\ x_2\geq 0,\,x_3=0\}$
intersects every $\s\o(2)$-orbit once. Thus, every revolution surface with that axis
can be obtained under the $\s\o(2)$-action of a curve on ${\cal P}$.

Let $\alpha(t)=(\cosh k(t),\sinh k(t),0,h(t))\subset{\cal P}$, where $k(t)\geq 0$
and $t$ the arc length of $\alpha$, that is, $k'(t)^2+h'(t)^2=1$. If $\alpha(t)$
generates a complete surface with positive extrinsic curvature then, from
Proposition \ref{p3.1}, this revolution surface $S$ must be topologically a sphere
or a graph on an open domain in $\h^2$ and, so, $\alpha(t)$ intersects the axis at
least once. In fact, the curve intersects the axis orthogonally; otherwise the
revolution surface would not be smooth.  But, from the vertical height estimates of
Section 6, there is no complete revolution $K-$surface with one end.

Now, $S$ can be parametrized by
$$
\psi(t,v)=(\cosh k(t),\sinh k(t)\cos v,\sinh k(t)\sin v,h(t)).
$$

A straightforward computation shows that the principal curvatures of $S$ are given by
\begin{eqnarray}
\lambda _1 &=& k' h'' - k'' h' \label{l1rev}\\
\lambda _2 &=& h' \coth k \label{l2rev}
\end{eqnarray}and so, its extrinsic curvature is
$$ K = h' \coth k ( k' h'' - k'' h' ) $$where $'$ denotes the derivative with respect to $t$.
Since $k'^2+h'^2=1$, we have $k' k'' + h' h'' =0$ and so
\begin{equation}\label{Krev1}
K = - k'' \coth k.
\end{equation}

Let us assume that $K$ is a positive constant, then
$$ (k')^2 = C_1 - 2 K \ln \cosh k $$where $C_1$ is a constant to be determined by the
boundary conditions.

As $\alpha$ must cut the axis orthogonally, we can assume that the lowest point
occurs at $t=0$.  Then $k(0)=0$ and $k'(0) =1 $, thus $C_1 =1$, that is
\begin{equation}\label{Krev2}
(k')^2 = 1 - 2 K \ln \cosh k .
\end{equation}

Since the lowest point occurs at $t=0$, $h' (t)> 0$ and hence, by (\ref{l2rev}),
$\lambda _2 (t)>0$ for sufficiently small $t$. But $\lambda _2 $ must have the same
sign, so $h'(t)>0$ for all $t$. Hence, $h$ increases as $t$ increases (and $S$ must
be embedded).

Suppose $z$ is the highest point of $S$, then $k (t_z) =0 $ for some $t_z >0 $.
Hence the domain of $t$ is $[0, t_z]$ and
$$
k(0)=k(t_z)=0,\quad h'(0)=h'(t_z)=0,\quad k'(0)=1,\quad k'(t_z)=-1.
$$
On the other
hand, by (\ref{Krev1}), $k'' < 0$ which implies that $k'$ decreases from $k'(0)=1$
to $k'(t_z)=-1$. So, as $t$ increases from $0$ to $t_z$, $k$ first increases from
$0$ to $k_{max}=k(t_{max})$, for some $t_{max} \in [0,t_z]$, then decreases from
$k_{max}$ to $0$.

Thus, $k$ must increase from $0$ to $t_{max}$ and $k'(t_{max})=0$. So, from
(\ref{Krev2}), one obtains $ k(t_{max}) = \cosh ^{-1} \exp \left( 1/2K\right) $.

Now, let $u= k' $, where $-1 \leq u \leq 1$, then by equation (\ref{Krev2}) we have
\begin{equation}\label{Krev3}
k = \cosh ^{-1} \exp \left( \frac{1-u^2}{2 K}\right)  , \quad -1 \leq u \leq 1 .
\end{equation}

Since $u= k'$, $\frac{d u}{ dt} = k'' = - K \tanh k$, we have by (\ref{Krev3})
$$
\frac{d h}{du} = -\frac{1}{K}\frac{\sqrt{1-u^2}}{\sqrt{1-\exp\left( -\frac{1-u^2}{K}
\right)}}.
$$

Then
\begin{equation}\label{hrev}
h = -\frac{1}{K}\int _1 ^{u}\frac{\sqrt{1-u^2}}{\sqrt{1-\exp\left( -\frac{1-u^2}{K}
\right)}} \,du + C , \quad -1 \leq u \leq 1
\end{equation}where $C$ is a real constant.

Also, we have
\begin{eqnarray*}
h_{min } &=& C\\
h_{max}  &=& \frac{1}{K}\int _{-1} ^{1}\frac{\sqrt{1-u^2}}{\sqrt{1-\exp\left(
-\frac{1-u^2}{K} \right)}} \,du + C
\end{eqnarray*}and
$$ h_0 = \frac{1}{K}\int _{-1} ^{0}\frac{\sqrt{1-u^2}}{\sqrt{1-\exp\left( -\frac{1-u^2}{K} \right)}} \,du + C $$

Therefore, from (\ref{Krev3}) and (\ref{hrev}), as $u$ decreases from $1$ to $-1$, $h$ increases from
$h_{min}$ to $h_{max}$, and $k$ first increases from $0$ to $k_{max}$ then decreases from $k_{max}$ to $0$.
Also, $S$ must be symmetric about $\h ^2 \times \set{h_0}$. Thus,

\begin{pro}\label{revh2r}
Let $S$ be a complete immersed sphere of revolution (about the vertical line
$\set{1,0,0}\times \r$) in $\hr$ with constant $K >0$, given by
$$\psi_K (u,v)=(\cosh k(u),\sinh k(u)\cos v,\sinh k(u)\sin v,h(u))$$where
$\alpha(u)=(\cosh k(u),\sinh k(u),0,h(u))$ is the generating curve of $S$.

Then, $S$ must be embedded and the generating curve is given by
\begin{eqnarray}
k(u) &=& \cosh ^{-1} \exp \left( \frac{1-u^2}{2 K}\right) \label{krev2}\\
h(u)&=& -\frac{1}{K}\int _1 ^{u}\frac{\sqrt{1-u^2}}{\sqrt{1-\exp\left( -\frac{1-u^2}{K} \right)}}
 \,du + C \label{hrev2}
\end{eqnarray}where $-1 \leq u \leq 1$ and $C$ is a real constant.

Also, $\h ^2 \times \set{h_0}$, with
$$ h_0 = \frac{1}{K}\int _{-1} ^{0}\frac{\sqrt{1-u^2}}{\sqrt{1-\exp\left( -\frac{1-u^2}{K} \right)}}
 \,du + C ,$$divides $S$ into two (upper and lower) symmetric parts.
\end{pro}

\begin{remark}
Let us observe that the above analysis is the same, in spirit, as in \cite{CR} for the case of revolution
surfaces in $\s^2 \times \r$ with constant positive extrinsic curvature. Moreover, we will need that
result, so we will state it here:
\begin{pro}\label{revsr}
Let $S$ be a complete immersed sphere of revolution (about the vertical line
$\set{1,0,0}\times \r$) in $\s ^2 \times \r\subset\r^4$ with constant $K
>0$, given by
$$\psi_K (u,v)=(\cos k(u),\sin k(u)\cos v,\sin k(u)\sin v,h(u))$$where
$\alpha(u)=(\cos k(u),\sin k(u),0,h(u))$ is the generating curve of $S$.

Then,
\begin{enumerate}
\item[i)] $S$ must be topologically a sphere and embedded.
\item[ii)] $S$ stays in $\d \times \r$, where $\d$ denotes the open hemisphere of $\s ^2$
of center $(1,0,0)$.
\item[iii)]the generating curve is given by
\begin{eqnarray}
k(u) &=& \cos ^{-1} \exp \left( -\frac{1-u^2}{2 K}\right) \label{ks2r}\\
h(u)&=&  -\frac{1}{K}\int _1 ^{u}\frac{\sqrt{1-u^2}}{\sqrt{\exp\left( \frac{1-u^2}{K} \right) -1}}
 \,du + C \label{hs2r}
\end{eqnarray}where $-1 \leq u \leq 1$ and $C$ is a real constant.
\end{enumerate}

Also, $\d \times \set{h_0}$, where
$$ h_0 = \frac{1}{K}\int _{-1} ^{0}\frac{\sqrt{1-u^2}}{\sqrt{\exp\left( \frac{1-u^2}{K} \right)-1}}
 \,du + C ,$$divides $S$ into two (upper and lower) symmetric parts.
\end{pro}
\end{remark}

\section{Vertical and horizontal height estimates for $K-$surfaces.}

We divide this Section in three parts. First, we establish some necessary equations
for surfaces with positive extrinsic curvature in $\m \times \r $, $\m$ a Riemannian
surface. Second, we obtain vertical height estimates for compact embedded surfaces
with constant positive extrinsic curvature in $\m\times \r $ and boundary in a
slice. And finally, we give horizontal height estimates for $K-$surfaces in $\hr$
and boundary on a vertical plane.

\subsection{Necessary equations.}

We will work in the spirit of \cite{AEG1} but using the conformal structure induced by the
second fundamental form of the surface as in \cite{AEG2}.

Let us denote by $g$ the metric of $\m $. Then the metric of $\m \times \r$ is given
by $\meta{}{}=g_{\kappa }+dt^2$. Let $\psi :S \longrightarrow \m \times \r$ be an
immersion with positive extrinsic curvature $K$.

Let $\pi :\m \times \r \longrightarrow\m  $ and $\pi_{\r}:\m  \times \r
\longrightarrow \r$ be the usual projections. We denote by $h:S\longrightarrow\r$
the height function, that is, $h(z)=\pi_{\r}(\psi(z))$, and $\nu=\meta{ N}{\pt}$,
$\pt$ the gradient in $\m  \times \r$ of the function $t$.

Since $K>0$ the second fundamental form $II$ is definite (and positive definite for
a suitable normal $N$). Then, we can choose a conformal parameter $z$ such that the
fundamental forms $I$ and $II$ can be written as
\begin{equation}\label{Iparametroz}
\begin{split}
 I & =  \meta{d\psi}{d\psi} =  E dz ^2 + 2 F \abs{dz}^2 + \ec d\bar{z}^2, \\
II & =  -\meta{d\psi}{dN} =  2 \rho \abs{dz}^2,\qquad\rho>0.
\end{split}
\end{equation}

Here
\begin{equation}\label{extrinseca}
K= - \frac{\rho ^2}{D} ,
\end{equation}
with $D= \abs{E}^2 - F^2 < 0$. The mean curvature of $S$ is
\begin{equation}
H=-\frac{F\rho}{D}=\frac{K}{\rho}\ F\label{HH}.
\end{equation}

Let us write
$$
\pt=T+\nu\,N
$$
where $T$ is a tangent vector field on $S$. Since $\pt$ is the gradient in $\m
\times \r$ of the function $t$, it follows that $T$ is the gradient of $h$ on $S$.
Thus, from (\ref{Iparametroz}), one gets
\begin{equation}\label{T}
T=\frac{1}{D}(\alpha \parz + \ac \parzc ),
\end{equation}
where
\begin{equation}
\alpha = \ec \hz - F \hzb \label{alpha} .
\end{equation}

In addition, we obtain the following equations
\begin{eqnarray}
\langle T,T\rangle &=& \frac{1}{D}( \alpha  h_z+ \ac \hzb) \label{TT} \\
\langle T,T\rangle &=& \frac{1}{D}(E\hzb^2+\bar{E}h_z^2-2F|h_z|^2)\label{TT2}\\
h_z &=& \frac{1}{D}( E \alpha + F \ac ) \label{hz}\\
|h_z|^2&=&||T||^2F+\frac{|\alpha|^2}{D}\label{hzhw}.
\end{eqnarray}

\begin{lema}\label{l6.1}
Let $\psi : S \longrightarrow \m \times \r $ be an immersion with $K>0$. Then, for a
conformal parameter $z$ for the second fundamental form, the following equations are
satisfied:
\begin{eqnarray}
\text{Codazzi} & \mbox{} & \frac{\rho _{\bar{z}}}{\rho} + ( \Gamma _{12}^1 -\Gamma
_{22}^2 ) = \kappa  \alpha \frac{\nu}{\rho} \label{codazzi}\\
 & \mbox{} & h_{zz} =\Gamma_{11}^1 \hz + \Gamma _{11}^2 \hzb \label{hzz}\\
& \mbox{} & h_{z\bar{z}} =\Gamma_{12}^1 \hz + \Gamma _{12}^2 \hzb + \nu \rho
 \label{hzw0} \\
 & \mbox{} & \nu _{\bar{z}}=  \frac{\alpha K }{\rho} \label{nuzb}
\\
& \mbox{} & ||T||^2+\nu^2=1.\label{nu2}
\end{eqnarray}

Here $\kappa (p)$ stands for the Gauss curvature of $\m$ at $\pi(\psi(p))$, and
$\Gamma^{k}_{ij}$, $\ i,j,k=1,2$, are the Christoffel symbols associated to $z$.
\end{lema}
\begin{proof}
From (\ref{Iparametroz}) we have
\begin{eqnarray}
\nabla_{\parz} \parz &=& \Gamma^1_{11}\parz+\Gamma^2_{11}\parzc \nonumber\\[2mm]
\nabla_{\parz} \parzc &=& \Gamma^1_{12}\parz+\Gamma^2_{12}\parzc +\rho \,N\label{dos}\\
-\nabla_{\parzc} N&=& \frac{K}{\rho} \left(\bar{E}\parz-F \,\parzc \right).\nonumber
\end{eqnarray}

Thus, the scalar product of these equalities with
$\pt$ gives us (\ref{hzz}), (\ref{hzw0}) and (\ref{nuzb}), respectively.

The last equation follows from
$$ 1=\langle\pt,\pt\rangle=\langle T,T\rangle+\nu^2 $$

Finally, from (\ref{dos}) we get
$$
\langle\nabla_{\parzc}\nabla_{\parz}\ \parz-\nabla_{\parz}\nabla_{\parzc}\ \parz ,N\rangle\ =\
\rho _{\bar{z}} + \rho ( \Gamma _{12}^1 -\Gamma _{22}^2 )
$$

Hence, using the relationship between the curvature tensors of a product manifold
(see, for instance, \cite[p. 210]{O}), the Codazzi equation becomes
$$ \kappa  \alpha \frac{\nu}{\rho} =\
\frac{\rho _{\bar{z}}}{\rho} + ( \Gamma _{12}^1 -\Gamma _{22}^2 ). $$

That is, (\ref{codazzi}) holds.
\end{proof}

\begin{remark}
The equation (\ref{nu2}) will be used subsequently  without comment.
\end{remark}

From now on we will assume that $K$ is a positive constant on $S$. A straightforward
computation gives us
\begin{equation}\label{Dzb}
\Gamma _{12}^1 + \Gamma _{22}^2 = \frac{D_{\bar{z}}}{2 D}.
\end{equation}

Thus, from (\ref{codazzi})
$$ \frac{D_{\bar{z}}}{2D} -\frac{\rho _{\bar{z}}}{\rho} = 2 \Gamma _{12}^1 -\kappa \alpha \frac{\nu
}{\rho}.
$$

Since $K$ is constant, we obtain from (\ref{extrinseca})
\begin{equation}\label{titirititi}
\frac{D_{\bar{z}}}{2D} -\frac{\rho _{\bar{z}}}{\rho} = 0
\end{equation}
and
\begin{equation}\label{g121}
\Gamma _{12}^1 = \kappa \alpha \frac{\nu }{2\rho}.
\end{equation}

Using $\Gamma _{12}^2=\overline{\Gamma _{12}^1}$, (\ref{g121}), (\ref{hz}) and
(\ref{extrinseca}), one has
\begin{eqnarray}
E_{\bar{z}} &= &2\meta{\nabla_{\partial_{\bar{z}}}\partial_{z}}{\partial_{z}}= 2(E \Gamma _{12}^1 + F \Gamma_{12}^2) \nonumber\\
  &=&\kappa \frac{\nu }{\rho}( E\alpha + F \ac )=\kappa \frac{\nu D}{\rho}h_z = -\kappa \frac{\nu \rho}{K}h_z
  \label{phi}.
\end{eqnarray}
On the other hand, by using (\ref{hzw0}), (\ref{g121}), (\ref{TT}) and
(\ref{extrinseca})
\begin{eqnarray}
2 h_{z\zb} &= &\kappa \frac{\nu}{\rho}(\alpha h_z+\ac \hzb)+2\nu\rho=\kappa
\frac{\nu}{\rho}D||T||^2+2\nu\rho=-\kappa \frac{\nu\rho}{K}(1-\nu^2)+2\nu\rho
\nonumber \\
&=& \frac{\nu\rho}{K}(2K-\kappa(1-\nu^2))\label{hzw}. 
\end{eqnarray}
Now, we compute $\nu_{z \zb}$. From (\ref{nuzb}) and (\ref{titirititi})
\begin{equation}\label{nzzb0}
 \nu _{z \zb}=\alpha _z\frac{K}{\rho} - \alpha K\frac{\rho _z}{\rho ^2}=
 \alpha _z\frac{K}{\rho}-\frac{\alpha K}{\rho}\frac{D_z}{2D}.
\end{equation}

Hence, we need to compute $\alpha _z$ :
$$ \alpha _z = 2 \meta{\nabla_{\partial_{z}}\partial_{\bar{z}}}{\partial_{\bar{z}}}\hz +
\ec h_{zz}-\meta{\nabla_{\partial_{z}}\partial_{z}}{\partial_{\bar{z}}}\hzb
-\meta{\partial_{z}}{\nabla_{\partial_{z}}\partial_{\bar{z}}}\hzb - F h_{z\zb}
$$
where we have used (\ref{alpha}).

If one considers (\ref{hzw0}) and (\ref{hzz}) then
\begin{eqnarray*}
2 \meta{\nabla_{\partial_{z}}\partial_{\bar{z}}}{\partial_{\bar{z}}}\hz  &=& 2 h_z F
\Gamma _{12}^1 + 2 \hz \ec \Gamma
_{12}^2 \\
\ec h_{zz} &=& h_z \ec \Gamma _{11}^1 +  \hzb \ec \Gamma _{11}^2\\
-\meta{\nabla_{\partial_{z}}\partial_{z}}{\partial_{\bar{z}}}\hzb &=& -\hzb F \Gamma _{11}^1 -  \hzb \ec \Gamma _{11}^2\\
-\meta{\partial_{z}}{\nabla_{\partial_{z}}\partial_{\bar{z}}}\hzb &=& -\hzb E \Gamma _{12}^1 -  \hzb F \Gamma _{12}^2\\
- F h_{z\zb} &=& -h_z F \Gamma _{12}^1 -  \hzb F \Gamma _{12}^2 - \nu \rho F.
\end{eqnarray*}

Therefore,
\begin{equation*}
\begin{split}
\alpha _z &= \alpha (\Gamma _{12}^2 + \Gamma _{11}^1) + \alpha \Gamma _{12}^2 - \ac
\Gamma_{12}^1 - \nu \rho F \\
 &= \alpha \frac{D_z}{2D} - \nu \rho F
\end{split}
\end{equation*}
where (\ref{Dzb}) is used  and that $\ac \Gamma_{12}^1$ is a real function from
(\ref{g121}).

Finally, using (\ref{nzzb0}) and (\ref{HH})
\begin{equation}\label{nzzb}
\nu _{z \zb} = -K F \nu=-\rho H\nu.
\end{equation}

Hence, we have obtained the Laplacian of $h$, $\nu$ and the derivative of the
$(2,0)-$part of $I$ with respect to $II$. That is,
\begin{lema}
Let $\psi : S \longrightarrow \m \times \r$ be an immersion with constant positive
extrinsic curvature $K$ on $S$, then
\begin{eqnarray*}
\Delta^{II}h &=& (2K-\kappa(1-\nu^2))\frac{\nu}{K}\\
\Delta ^{II} \nu  &=&-2H \nu\\
E_{\zb} &=& -\kappa \frac{\nu \rho}{K}h_z.
\end{eqnarray*}
\end{lema}
Now, we define a quadratic form which will play an important role in the following
sections. Let $\eps$ be a constant equal to $1$ or $-1$. Then, we consider the new
quadratic form
\begin{equation}\label{A}
A= I + g(\nu)\,dh^2
\end{equation}where $g(\nu)$ is the only solution to the ODE
$$
\phi(\nu)=\frac{\nu^2-1}{2}g'(\nu)
$$
such that it is well-defined for $\nu=\pm1$, where
\begin{eqnarray*}
\phi&=&\frac{\nu}{K}\left( (2K+\eps(\nu^2-1)) g - \eps  \right) -(1-\nu^2)g'.
\end{eqnarray*}

That is, $g(\nu)$ is the real analytic function given by
\begin{equation}\label{gdenu}
g(\nu)=\frac{\nu^2-1+\eps K(e^{\frac{\eps(1-\nu^2)}{K}}-1)}{(1-\nu^2)^2}.
\end{equation}

\begin{remark}\label{Remark1} In order to prove the above assertion observe that
the function $\widetilde{g}:\r\longrightarrow\r$
\begin{equation*}
\widetilde{g}(t)=\frac{t^2-1+\eps K(e^{\frac{\eps(1-t^2)}{K}}-1)}{(1-t^2)^2}
\end{equation*}
is well defined for each $t\in\r$ (in particular, if $t=\pm 1$) and it is an
analytic function. This is easy to see bearing in mind that
$$
e^t=1+\sum_{n=1}^{\infty}\frac{t^n}{n\mbox{!}}.
$$
\end{remark}

In addition
$$
\frac{\eps}{2K}=g(\pm 1)\leq g(\nu)\leq g(0)=-1+\eps K(e^{\frac{\eps}{K}}-1).
$$

Let us observe that
\begin{equation}\label{fi1}
\chi(\nu)=1+g(\nu)||T||^2=\frac{\eps K(e^{\frac{\eps(1-\nu^2)}{K}}-1)}{1-\nu^2}
\end{equation}
satisfies
\begin{eqnarray}
1=\chi(\pm 1)\leq \chi(\nu)&\quad\mbox{if}\quad&\eps=1\nonumber\\
0<K(1-e^{\frac{-1}{K}})=\chi(0)\leq \chi(\nu)&\mbox{if}&\eps=-1\label{fi2}
\end{eqnarray}
for all $-1\leq\nu\leq 1$.

Let us denote
\begin{equation}\label{Q}
Q = E + g(\nu)\hz ^2,
\end{equation}
then $Qdz^2$ can be considered as the $(2,0)-$part of the real quadratic form $A$
for the second fundamental form $II$.

The {\it extrinsic curvature of the pair} $(II,A)$ (see \cite{Mi}) is given by
\begin{eqnarray}
K(II,A)&=&\frac{(F + g\,|h_z|^2)^2-|Q|^2}{\rho^2}\nonumber\\
&=&\frac{(F^2-|E|^2)+g(2F|h_z|^2-E\hzb^2-\bar{E}h_z^2)}{\rho^2}\label{KIIA}\\
&=&\frac{1}{K}+\frac{g}{K}||T||^2=\frac{1}{K}(1+g||T||^2),\nonumber
\end{eqnarray}
where we have used (\ref{extrinseca}) and (\ref{TT2}).

In particular, the previous computation gives us
\begin{equation}\label{modQ}
|Q|^2=( F+g \,|h_z|^2)^2 + D(1+g\norm{T}^2).
\end{equation}

\subsection{Vertical height estimates in M$^2$xR.}

Here, we establish upper bounds for compact graphs with positive constant extrinsic
curvature and boundary in a slice of a product space $\m\times\r$.

\begin{teo}\label{testver}
Let $\m$ be a Riemannian surface, $\psi : S \longrightarrow \m \times \r$ be a
compact graph on a domain $\Omega \subset \m $, with positive constant extrinsic
curvature $K$ and whose boundary is contained in the slice $\m \times \set{0}$. Let
$k$ be the minimum of the Gauss curvature on $\Omega \subset \m$. Then, there exists
a constant $c_K $ (depend only on $K$ and $k$) such that $h(p) \leq c_K$ for all $p
\in S$.
\end{teo}
\begin{proof}
We want to compute the Laplacian of a certain function given by
$h+f(\nu)$ for a suitable real function $f$. Since, we know $h_{z\bar{z}}$
from (\ref{hzw}) then we focus our attention on $f(\nu)_{z\bar{z}}$.

By using (\ref{nuzb}), (\ref{nzzb}), (\ref{extrinseca}) and (\ref{hzhw})
\begin{eqnarray*}
f(\nu)_{z\bar{z}}&=&f'(\nu)\nu_{z\bar{z}}+f''(\nu)|\nu_{z}|^2=
-f'(\nu)KF\nu+f''(\nu)\frac{|\alpha|^2K^2}{\rho^2}\\
&=& -f'(\nu)KF\nu-f''(\nu)K(|h_z|^2-(1-\nu^2)F)\\[2mm]
&=&-K\left(F(\nu f'(\nu)-(1-\nu^2)f''(\nu))+|h_z|^2f''(\nu)\right).
\end{eqnarray*}

First, note that since $II$ must be positive definite and $S$ is a graph, then $\nu\leq 0$.
Second, we distinguish two cases, $k=0$ and $k \neq 0$. When $k=0$, we consider
$$ f(\nu ) = \frac{\nu}{\sqrt{K}} $$thus
$$
f(\nu)_{z\bar{z}} = - \sqrt{K}F \nu
$$
and, from (\ref{extrinseca}),
\begin{equation}\label{estik0}
\begin{split}
(h+f(\nu))_{z\bar{z}}& =\frac{2K-\kappa(1-\nu^2)}{2K}\nu\rho - \sqrt{K}F \nu \\
&=- \frac{\kappa}{2K} (1-\nu ^2) \nu \rho +
\left(\sqrt{-D}-F\right)\nu\sqrt{K}\\
   & \geq - \frac{\kappa}{2K} (1-\nu ^2) \nu \rho \geq 0
\end{split}
\end{equation}thus, one has $\Delta^{II}(h+f(\nu))\geq 0$ on our surface and $h+f(\nu)\leq 0$ on
the boundary, so $h \leq - \nu /\sqrt{K} \leq 1/ \sqrt{K} $.

When $k \neq 0$, we can suppose that $k = \eps$, where $\eps$ is $-1$ or $1$. To do
that it is enough to consider, the new metric on $\m \times \r$ given by the
quadratic form $|k|\,g+dt^2$ and the surface $S '=\{(x,\sqrt{|k|}\,t)\in\m \times
\r:\ (x,t)\in S \}$ which has constant extrinsic curvature $K/|k|$. Here, $g$
denotes the induced metric on $\m$.

We consider
$$
f'(\nu)=\sqrt{\eps\frac{1 -  e^{-\eps\frac{1-\nu^2}{K}}}{1-\nu^2}}.
$$
This function is real analytic and so is every primitive $f(\nu)$.

In addition
$$
f''(\nu)=\frac{\nu e^{-\eps \frac{1-\nu^2}{K}}}{K}\ \frac{g(\nu)}{f'(\nu)}
$$
where $g(\nu)$ is given by (\ref{gdenu}).

Thus,
$$
\nu f'(\nu)-(1-\nu^2)f''(\nu)=\frac{\nu \,e^{-\eps\frac{1-\nu^2}{K}}}{K\,f'(\nu)}
$$
and
$$
f(\nu)_{z\bar{z}}=-\ \frac{\nu
\,e^{-\eps\frac{1-\nu^2}{K}}}{ f'(\nu)}\,(F+g(\nu)|h_z|^2).
$$

We observe that $(F+g(\nu)|h_z|^2)|dz|^2$ is the (1,1)-part of the quadratic form
$A$, given by (\ref{A}), for the second fundamental form and our
quadratic form $Q dz^2$, given by (\ref{Q}), is the (2,0)-part of $A$ for $II$.

Moreover, from (\ref{hzw}) and (\ref{KIIA})
$$
(h+f(\nu))_{z\bar{z}}=\frac{2K-\kappa(1-\nu^2)}{2K}\nu\rho-\, \frac{\nu
\,e^{-\eps\frac{1-\nu^2}{K}}}{f'(\nu)}\,\sqrt{\frac{\rho^2(1+g(\nu)(1-\nu^2))}{K}+|Q|^2}
$$
Here, we have used that $F+g(\nu)|h_z|^2>0$. This fact is clear because $K(II,A)$ is
positive from (\ref{KIIA}), (\ref{fi1})  and (\ref{fi2}), so, $A$ is positive
definite or negative definite. Thus, $F+g(\nu)|h_z|^2$ is positive at every point or
is negative everywhere. But, it is clear that it is positive at a highest point
($h_z=0$ at this point).

Hence,
\begin{eqnarray*}
(h+f(\nu))_{z\bar{z}}&\geq& \frac{2K-\kappa(1-\nu^2)}{2K}\nu\rho-\, \frac{\nu
\,e^{-\eps\frac{1-\nu^2}{K}}}{f'(\nu)}\,\sqrt{\frac{\rho^2(1+g(\nu)(1-\nu^2))}{K}}\\
&=&\left(\frac{2K-\kappa(1-\nu^2)}{2K}-e^{-\eps\frac{1-\nu^2}{2K}}\right)\nu\rho\\
&\geq&-\left(e^{-\eps \frac{1-\nu^2}{2K}}-1+\eps
\frac{1-\nu^2}{2K}\right)\nu\rho\geq0.
\end{eqnarray*}
Here, we use that the term between parenthesis is non-negative. This is because the
real function $e^t-1-t$ is non-negative everywhere.

By taking,
$$
f(\nu)=\int_0^{\nu}f'(t)dt
$$
one has $\Delta^{II}(h+f(\nu))\geq 0$ on the surface and $h+f(\nu)\leq 0$ on
the boundary (because $f'(\nu)\geq 0$ and $\nu\leq 0$).

Hence, the maximum height is less than or equal to
$$
c_K:=\int_{-1}^0 f'(t) dt.
$$
\end{proof}

\begin{remark}
It is clear that the height estimate $c_K$, when $k\neq 0$, is not reached for any graph with
positive constant $K$ and boundary on a slice since
$$
e^{-\eps\frac{1-\nu^2}{2K}}-1+\eps \frac{1-\nu^2}{2K}
$$
is positive for $\nu\neq-1$ as well as $\Delta^{II}(h+\phi(\nu))$. But then the maximum principle at the highest point shows $\Delta^{II}(h+\phi(\nu))$ vanishes identically; a contradiction.

But, when $k=0$, by (\ref{estik0}), if the maximum height is attained at a point,
then $h - \nu/ \sqrt{K}$ vanishes identically on $S$. Thus, using (\ref{estik0})
again, $\kappa$ and $E$ vanish identically. That is, the domain $\Omega$ is flat and
$S$ is totally umbilical.
\end{remark}

As a standard consequence of the Alexandrov reflection principle for surfaces of
constant mean curvature with respect to the slices $\m \times \{t_0\}$, we have the
following Corollary

\begin{corolario}\label{estver}
Let $\psi : S \longrightarrow \m\times \r$ be a compact embedded surface with
positive constant extrinsic curvature $K$ and whose boundary is contained in the
slice $\m\times \set{0}$.  Let $k$ be the infimum of the Gauss curvature on $\m$.
Then, there exists a constant $c_K $ (depend only on $K$ and $k$) such that $h(p)
\leq 2 c_K$ for all $p \in S$.
\end{corolario}

We also observe that if $S$ is a non-compact properly embedded $K-$surface without
boundary in $\m\times\r$ and $\m$ is compact then $S$ must have at least one top end
and one bottom end. This is a consequence of our height estimates (see, for
instance, \cite{HLR}).

\subsection{Horizontal height estimates.}

Now, we consider a compact embedded $K-$surface in $\hr$ with boundary on a vertical
plane and obtain distance estimates to this plane.

\begin{teo}\label{testhor}
Let $S$ be a compact embedded surface in $\hr$, with extrinsic curvature a constant $K>0$. Let $P$
be a vertical plane in $\hr$ and assume $\partial S \subset P$. Then the distance from $S$ to $P$
is bounded; i.e., there is a constant $d$, independent of $S$, such that
$$ \mbox{dist}(q , P) \leq d \, , \quad \forall q \in S .$$
\end{teo}
\begin{proof}
Let $q\in S$ be a furthest point from $P$. Up to isometry, we can assume $q \in \hc$
and $q \in \mbox{ext}_{\hr}(P)$. Let $P_{\gamma}(t)$ be the foliation of vertical
planes along $\gamma$ with $P_{\gamma}(0)=P$ and $q \in P_{\gamma}(h)$. Let $X$
denote the horizontal Killing field of $\hr$ generated by translations along $\gamma
$ ($X$ is tangent to each $\h ^2 \times \set{\tau}$ and is translation along $\gamma
\times \set{\tau}$); $X$ is orthogonal to the planes $P_{\gamma}(t)$.

Now, do Alexandrov reflection with the planes $P_{\gamma}(t)$, starting at $t = h$, and decrease $t$.
For $h/2<t \leq h$, the symmetry of the part $S$ in $\mbox{ext}_{\hr}(P_{\gamma}(t))$, does not
touch $\partial S$, since $\partial S \subset P$. Hence the Alexandrov reflection technique shows
that the symmetry of $S^+ (t)= S \cap \mbox{ext}_{\hr}(P_{\gamma}(t))$, by $P_{\gamma}(t)$, intersects
$S$ only at $S \cap P_{\gamma}(t)$ and $S$ is never orthogonal to $P_{\gamma}(s)$ for $t \leq s \leq h$.
Since $X$ is orthogonal to each $P_{\gamma}(t)$, we conclude $X$ is transverse to $S^+ (h/2)$, and
$S^+ (h/2)$ is a graph over a domain of $P_{\gamma}(h/2)$ with respect to the orbits of $X$.

Thus, to prove the Theorem, it suffices to prove that $X-$graphs are a bounded distance from $P$, assuming
the boundary of the graph is in $P$.

Now, suppose $S$ is an $X-$graph over a domain $D \subset P$ and chose $P_{\gamma } (t)$ as before. Let
$S_R$ be the rotationally invariant sphere whose extrinsic curvature $K$ is the same as that of $S$.
Denote by $c=c(K)$, the diameter of $S_R$.

We will now prove that for each $t > 2c $, the diameter of each connected component of $S(t) = P(t)\cap S$,
is at most $2c$. Suppose not, so for some component $C(t)$ of $S(t)$, there are points $x,y$ inside
the domain $D(t)$ of $P(t)$ bounded by $C(t)$ with $\mbox{dist}(x,y)> 2c$. Let $Q$ be the bounded domain
of $\hr$ bounded by $S \cup D$. Let $\beta$ be a path in $D(t)$ joining $x$ to $y$, $\beta$ is disjoint from $C(t)$.
Let $\Omega$ be the ``rectangle'' formed by the orbits of $X$ joining $\beta$ to $P$;
$\Omega \subset Q$. Let $p$ be a point of $\Omega $ whose distance to $\partial \Omega $ is
greater than $c$; $p$ exists by construction of $\Omega$.

Let $\eta$ be the geodesic through $p$, each of whose points is a distance greater
than $c$ from $\partial \Omega$; it is easy to find such an $\eta$ is in the plane
$P(t)$ containing $p$. $\eta$ ``enters'' $Q$ at a first point $q_0$ and ``leaves''
$Q$ at a last point $q_1$.

Now, consider the family of spheres centered at each point of $\eta$,
each sphere obtained from the rotational sphere $S_R$ (of extrinsic curvature $K$)
by a translation of $\hr$. Consider the family of spheres as entering $Q$ at $q_0$.

Then, there is some first sphere in the family (coming from $q_0$) that touches $\Omega$
for the first moment at an interior point of $\Omega$. Then the sphere passes through $\Omega$,
not touching $\partial Q$ initially, and the sphere passes through $\Omega$ without touching
$\partial \Omega$. Since the spheres leave $Q$ at $q_1$ there is some sphere that touches
$\partial Q \cap S$ at a first point of contact of the spheres, that passed through $\Omega$,
with $S$. At this first point of contact, the mean curvature vectors of $S$ and the rotational
sphere are equal.  Hence $S$ equals this sphere by the maximum principle; a contradiction.

Now, if the Theorem is false, there is a sequence of graphs $S_n$ over domains $D_n
\subset P$, with $\mbox{diameter}(D_n) < 2c$ and $\mbox{dist}(S_n , P)$ unbounded.

After an ambient isometry we can assume the $D_n$ are contained in a fixed disk $D$
and the $S_n$ are contained in the horizontal Killing cylinder $\mathcal{C}$ over
$D$, a tubular neighborhood of a horizontal geodesic $\gamma $. We will use
``tilted'' vertical planes to show this is impossible.

We  can assume, without lost of generality, that $\gamma =\set{ 0,\pi}$, $P$ is the
vertical plane over the geodesic $\set{ \pi /2,3 \pi /2}$, and the graphs  $S_n$
satisfy  $\pi(S_n)$ are asymptotic to 0.  Consider the   vertical plane
$Q(s_1,s_2)=\set{ s_1,s_2} \times \r$. A simple calculation shows that for $s_1 = 0
$ and $s_2$ positive and close to $0$, then   the symmetry through $Q(0,s_2)$ of
$\mathcal{C} \cap \mbox{int}_{\hr} (Q(0,s_2))$ does not intersect $P$. In
particular, the symmetry of the part of any $S_n \cap \mbox{int}_{\hr} (Q(0,s_2))$,
does not intersect $\partial S_n \subset D$.  By continuity, for $s_1$ negative and
sufficiently close to 0, the above last two statements continue to hold, i.e., the
symmetry through $Q(s_1,s_2)$ of $\mathcal{C} \cap \mbox{int}_{\hr} (Q(s_1,s_2))$
does not intersect $P$ and the symmetry of the part of any $S_n \cap
\mbox{int}_{\hr} (Q(s_1,s_2))$, does not intersect $\partial S_n \subset D$.

Now observe that the symmetry of $\mathcal{C} \cap \mbox{int}_{\hr} (Q(s_1,s_2))$
goes outside $\mathcal{C}$, hence  the symmetry of $S_n \cap \mbox{int}_{\hr}
(Q(s_1,s_2))$ also goes outside  $\mathcal{C}$, for $n$ sufficiently large.

Choose $a$ and $b$ between $s_1$ and $s_2$ so that $Q(a,b)$ is disjoint from
$\mathcal{C}$.  Let $R(t)$ be a foliation by vertical planes of the region of $\hr$
between $Q(a,b)$ and $Q(s_1,s_2)$ with $R(0) = Q(a,b)$ and $R(1) = Q(s_1,s_2)$, $0
\leq t \leq 1$.

Consider doing Alexandrov reflection with the planes $R(t)$.  Choose $n$ large so
that the symmetry of $\ (S_n) \cap \mbox{int}_{\hr} (Q(s_1,s_2))$ has points outside
$\mathcal{C}$.  The symmetry of this part of $S(n)$ through each $R(t)$ is disjoint
from  $P$, hence disjoint from  $\partial S_n $. Also the symmetry of this part of
$S(n)$ through $R(1)$ goes outside  $\mathcal{C}$, and $R(0)$ is disjoint from
$\mathcal{C}$.  Hence there is a smallest $t$ such that the symmetry of $S(n)$
through $R(t)$ touches $S(n)$ at some point.  Thus $R(t)$ is a symmetry plane of
$S(n)$, which is a contradiction.  This completes the proof.
\end{proof}

This proof also works for properly embedded surfaces with constant mean curvature
greater than $1/2$. Thus, \cite[Theorem 1.2]{NR} can be extended for $H-$surfaces
with $H> 1/2$, that is,

\begin{corolario}\label{nuevocorolario}
Let $H >1/2$ and let $S$ be a properly embedded $H-$surface in $\hr$ with finite
topology and one end. Then $S$ is contained in a vertical cylinder of $\hr$.
\end{corolario}

\section{Classification of complete $K-$surfaces in $\hr$ and $\s ^2 \times \r$.}

In this Section $\m (\eps)$ stands for $\s ^2$ or $\h ^2 $, depending on $\eps =1$
or $\eps = -1 $, respectively. We continue working with a conformal parameter $z$
for the second fundamental form $II$. We now come to a  key Lemma.

\begin{lema}
Let $\psi : S \longrightarrow \m (\eps)\times \r$ be an immersed $K-$surface. If we consider $S$ as the Riemann
surface with the conformal structure induced by its second fundamental form, then the quadratic
form given by (\ref{Q}) verifies
\begin{equation}
|Q_{\zb}|^2\leq \frac{K g'(\nu)^2(1-\nu
^2)^2|h_z|^2}{4\chi(\nu)}|Q|^2.\label{QzleqQ}
\end{equation}where $g$ and $\chi$ are given by (\ref{gdenu}) and (\ref{fi1}) respectively.
\end{lema}
\begin{proof}
First, we compute the derivative of $Q$, which is given by
\begin{equation*}
Q_{\zb} =  E_{\zb} + g' \nzb \hz ^2 +   2 g \hz h_{z\zb}.
\end{equation*}
So, from (\ref{phi}) and (\ref{hzw}), one gets
$$
E_{\zb} + 2 g \hz h_{z\zb}= \frac{\nu}{K}\left( (2K+\eps(\nu^2-1)) g - \eps
\right)\rho \hz.
$$
And (\ref{nuzb}), (\ref{hz}) and (\ref{TT}) give
$$
\nzb E = -\rho \frac{E \alpha}{D} = \rho \frac{\ac}{D} F - \rho h_z,
$$
$$
\nzb h_z^2 = - \rho h_z \frac{\alpha h_z}{D}= \rho \frac{\ac}{D} |h_z|^2 -
\norm{T}^2 \rho h_z,
$$
$$
 g' \nzb \hz ^2 = g' |h_z|^2  \rho \frac{\ac}{D}- g'\norm{T}^2  \rho h_z.
$$

Consequently,
\begin{equation}\label{Qzb}
Q_{\zb} = \rho g'(\nu)\left( \frac{\nu^2-1}{2}h_z + |h_z|^2
\frac{\ac}{D}\right).
\end{equation}

So,
\begin{equation*}
\begin{split}
|Q_{\zb}|^2 &= \rho ^2g'(\nu)^2\left( (1-\nu^2)^2\frac{|h_z|^2}{4} -
(1-\nu^2)\frac{|h_z|^2}{2}(\frac{\alpha h_z + \ac h_{\zb}}{D})+|h_z|^4
\frac{|\alpha|^2}{D^2}\right)\\
 &=\rho ^2g'(\nu)^2|h_z|^2\left( -\frac{(1-\nu^2)^2}{4}
+|h_z|^2 \frac{|\alpha|^2}{D^2}\right)\\
 &=\frac{\rho ^2g'(\nu)^2|h_z|^2}{-D}\left( D\frac{(1-\nu^2)^2}{4}
+|h_z|^2 \frac{|\alpha|^2}{-D}\right)\\
&= K g'(\nu)^2|h_z|^2\left( D\frac{(1-\nu^2)^2}{4} +|h_z|^2 ((1-\nu ^2)F
-|h_z|^2)\right),
\end{split}
\end{equation*}
where we have used (\ref{TT}), (\ref{extrinseca}) and (\ref{hzhw}).

Thus, from (\ref{modQ})
$$
D= \frac{|Q|^2- ( F+g(\nu) \,|h_z|^2)^2 }{\chi(\nu)}.
$$

Hence, from the previous equation of $|Q_{\zb}|^2$
\begin{equation*}
\begin{split}
|Q_{\zb}|^2 &=K g'(\nu)^2|h_z|^2\left( \frac{|Q|^2- ( F+g(\nu) \,|h_z|^2)^2
}{\chi(\nu)}\
\frac{(1-\nu^2)^2}{4} +|h_z|^2 ((1-\nu ^2)F -|h_z|^2)\right)\\
 &=\frac{K g'(\nu)^2(1-\nu ^2)^2|h_z|^2}{4\chi(\nu)}|Q|^2+\\
 &-\frac{K g'(\nu)^2|h_z|^2}{4\chi(\nu)}\left( (1-\nu^2)^2(F+g(\nu) \,|h_z|^2)^2 -4 \chi(\nu)
  |h_z|^2 ((1-\nu ^2)F -|h_z|^2) \right)
\end{split}
\end{equation*}

Now, we show that the last term between parenthesis is non negative. That is,
$$\begin{array}
{l}
((1-\nu^2)F-(2+g(\nu)(1-\nu^2))|h_z|^2)^2=\\[2mm]
=(1-\nu^2)^2F^2-2F|h_z|^2(2(1-\nu^2)+g(\nu)(1-\nu^2)^2)+(4+4g(\nu)(1-\nu^2)+g(\nu)^2(1-\nu^2)^2)|h_z|^4\\[2mm]
=(1-\nu^2)^2(F^2-2Fg(\nu)|h_z|^2+g(\nu)^2|h_z|^4)+(1-\nu^2)|h_z|^2(-4F+4g(\nu)|h_z|^2)+4|h_z|^4\\[2mm]
=(1-\nu^2)^2(F+g(\nu)|h_z|^2)^2+4|h_z|^2(-Fg(\nu)(1-\nu^2)^2-F(1-\nu^2)+g(\nu)(1-\nu^2)|h_z|^2+|h_z|^2)\\[2mm]
=(1-\nu^2)^2(F+g(\nu)|h_z|^2)^2-4 \chi(\nu)
  |h_z|^2 ((1-\nu ^2)F -|h_z|^2)
\end{array}
$$
as we wanted to prove.

Therefore,
\begin{equation*}
|Q_{\zb}|^2\leq \frac{K g'(\nu)^2(1-\nu
^2)^2|h_z|^2}{4\chi(\nu)}|Q|^2.
\end{equation*}
\end{proof}

This Lemma shows that the \cite[Main Lemma]{ACT} can be used:

\begin{lema}
Let $\psi : S \longrightarrow \m (\eps) \times \r$ be an immersion with positive constant
extrinsic curvature. Then, the zeroes of $Q$ are isolated with negative index or $Q$
vanishes identically.
\end{lema}

As a consequence one has

\begin{teo}\label{Q=0}
Let $\psi : S \longrightarrow \m (\eps) \times \r $ be an immersion with positive constant
extrinsic curvature, with $S$ a topological sphere. Then $Q$ vanishes identically on
$S$.
\end{teo}

From Theorem \ref{t3.1} we know that every complete $K-$surface in $\hr$ is properly
embedded and it is compact or homeomorphic to $\r^2$. So, we will show that it
cannot be homeomorphic to $\r ^2$. This follows from the following result.

\begin{teo}\label{lnor2}
For $K>0$ (or $H > 1/2$) there is no properly embedded $K-$surface ($H-$surface) in
$\hr$ with finite topology and one end.
\end{teo}
\begin{proof}
We only outline the main steps of the proof since, in essence, it is the same as in
\cite{NR} for surfaces of constant mean curvature greater than $1/\sqrt{3}$.

Note that the Plane Separation Lemma is valid for properly embedded surfaces with
constant positive extrinsic curvature, so, Theorem \ref{testhor} ensures us that
such a surface must be contained in a vertical cylinder. Thus the Alexandrov
reflection method using horizontal planes and Theorem \ref{testver} completes the
proof as in \cite[Theorem 1.1]{NR}.

Analogously, the result is also valid for $H-$surfaces, $H>1/2$, using height
estimates in \cite{AEG1} and Corollary \ref{nuevocorolario}.
\end{proof}

With all of this,

\begin{teo}\label{tKconstant}
A complete immersion with positive constant extrinsic curvature $K$ in $\m
(\eps) \times \r$ is a rotational sphere (cf. Section 5).
\end{teo}
\begin{proof}
By the Gauss equation \cite{D}, the Gauss curvature $K(I)$ of the surface
satisfies
$$
K(I)=K+\varepsilon\nu^2.
$$
Thus, for $\varepsilon=1$, $K(I)\geq K>0$ and, so, every complete $K-$surface in $\s
^2 \times \r$ must be a topological sphere from the Bonnet and Gauss-Bonnet
theorems. On the other hand, from Theorem \ref{t3.1} and Theorem \ref{lnor2}, we can
also state that every complete $K-$surface in $\hr$ must be a topological sphere.

Thus, by Theorem \ref{Q=0}, $Q=0$ on any complete immersion with positive constant
extrinsic curvature $K$ in $\m (\eps) \times \r$.

Now, we show that the immersion is rotational. Let us take doubly orthogonal
coordinates $(u,v)$ for $(I,II)$, that is,
\begin{eqnarray*}
I&=&mdu^2+ndv^2\\
II&=&k_1mdu^2+k_2ndv^2,
\end{eqnarray*}
then since the metric $II$ and the real quadratic form $A$ given by (\ref{A}) are
conformal (because $Q$ is the $(2,0)-$part of $A$ for $II$) then $h_u h_v\equiv 0$.

Thus, we can assume that $h_u$ vanishes locally.  Using the compatibility equations given
in \cite{D} and making a suitable change of doubly orthogonal parameters, as in \cite[Theorem 3.1]{AEG2},
one sees that all $m,n,k_1,k_2,h,\nu$ only depend on the second parameter $v$. Therefore, the
uniqueness part in \cite{D} gives that the immersion is invariant under a
1-parameter group of transformations ($(u,v)\rightarrow(u+t,v)$). To finish, as the surface
is compact, then it must be invariant by the group of rotations, so Proposition \ref{revh2r} and
Proposition \ref{revsr} give us the result.
\end{proof}

Note that Theorem \ref{t3.1} together with Theorem \ref{lnor2} shows that a complete
$K-$surface $S$ in $\hr$ must be embedded  and topologically a sphere. Then, the
Alexandrov reflection principle with respect to vertical planes proves $S$ is a
rotational sphere. This gives us an alternative proof to Theorem \ref{tKconstant} in
$\hr$.

Observe that a similar reasoning does not seem possible in $\s^2 \times \r$. That
is, the existence of the quadratic form $Q$ with isolated zeroes of negative index
appears to be essential in this case.

\vspace{.5cm}

{\bf Concluding remarks:}\\[-2mm]

It would be interesting to know which results of this paper extend to the other
complete, simply connected, homogeneous 3-manifolds. For example, does the Hadamard-
Stoker theorem hold in Heisenberg space, the Berger Spheres or the universal cover
of $\mathbb{PSL}(2 , \mathbb{R})$? The space ${\rm Sol}_3$ is foliated by totally
geodesic surfaces, each isometric to $\h ^2$ (in fact, there are two such orthogonal
foliations, and their intersection is an Anosov flow). Using the above techniques,
it is not hard to see that immersed compact surfaces of positive extrinsic curvature
in ${\rm Sol} _3$ are embedded spheres. Is this the case in Heisenberg space? If the
extrinsic curvature $K$ is a positive constant, is the surface a rotational sphere?

{\small J.M. Espinar and J.A. Gálvez were partially supported by Ministerio de
Educación y Ciencia Grant No. MTM2004-02746 and Junta de Andalucía Grant No.
FQM325.}


\begin{thebibliography}{99999}\small

\bibitem[AR1]{AR} U. Abresch and H. Rosenberg, A Hopf Differential for
Constant Mean Curvature Surfaces in $\s^2\times \r$ and $\h^2\times \r$, {\it Acta
Math.} {\bf 193} (2004), 141--174.

\bibitem[AR2]{AR2} U. Abresch and H. Rosenberg, Generalized Hopf differentials,
{\it Mat. Contemp.} {\bf 28} (2005), 1--28.

\bibitem[AEG1]{AEG} J.A. Aledo, J.M. Espinar and J.A. G\'{a}lvez,
Complete surfaces of constant curvature in $\hr$ and $\s ^2 \times \r$, {\it Calc.
Variations \& PDEs} {\bf 29} (2007), 347--363.

\bibitem[AEG2]{AEG1} J.A. Aledo, J.M. Espinar and J.A. G\'{a}lvez,
Height estimates for surfaces with positive constant mean curvature in $\m \times
\r$. To appear in \emph{Illinois J. Math.}

\bibitem[AEG3]{AEG2} J.A. Aledo, J.M. Espinar and J.A. G\'{a}lvez,
Surfaces with Constant Curvature in $\s ^2\times \r$ and $\hr $. Height Estimates
and Representation. To appear in \emph{Bull. Braz. Math. Soc.}

\bibitem[ACT]{ACT} H. Alencar, M. do Carmo and R. Tribuzy, A theorem of H. Hopf and the
Cauchy-Riemann inequality. To appear in \emph{Comm. Anal. Geom.}

\bibitem[CW]{CW} M. P. do Carmo and F. W. Warner,
Rigidity and convexity of hypersurfaces in spheres, {\it J. Diff. Geom.} {\bf 4}
(1970), 133--144.

\bibitem[CR]{CR} X. Cheng and H. Rosenberg,
Embedded positive constant $r$-mean curvature hypersurfaces in $\mathbb{M}^
m\times\r$, {\it An. Acad. Brasil. Cienc.} {\bf 72} (2005), 183--199.

\bibitem[C]{C} R.J. Currier, On Hypersurfaces of Hyperbolic Space Infinitesimally
Supported by Horospheres, {\it Trans. Am. Math. Soc.} {\bf 313} (1989), 419--431.

\bibitem[D]{D} B. Daniel, Isometric immersions into
3-dimensional homogeneous manifolds, {\it Comment. Math. Helv.}, {\bf 82} (2007),
87--131.

\bibitem[H]{H} J. Hadamard, Sur certaines proprietes des trajectoires en dynamique,
\emph{J. Math. Pures Appl.} {\bf 3} (1897), 331--387.

\bibitem[HLR]{HLR} D. Hoffman, J.H.S. de Lira and H. Rosenberg,
Constant mean curvature surfaces in $\m\times\r$, {\it Trans. A.M.S.} {\bf 358}
(2006), 491--507.

\bibitem[Mi]{Mi} T.K. Milnor,
Abstract Weingarten Surfaces, {\it J. Diff. Geom.} {\bf 15} (1980), 365--380.

\bibitem[NR]{NR} B. Nelli and H. Rosenberg,
Simply connected constant mean curvature surfaces in $\hr$, {\it Michigan Math. J.}
{\bf 54} (2006), 537--543.

\bibitem[O]{O} B. O'Neill,
{\it Semi-Riemannian Geometry}, Academic Press, 1983.

\bibitem[S]{S} J. Stoker, Über die Gestalt der positiv gekrümmten offenen
Flächen im dreidimensionalen Raume, {\it Compositio Math.} {\bf 3} (1936), 55--88.




\end{thebibliography}
\end{document}